\documentclass[12pt,a4paper]{amsart}
\usepackage{amssymb}

\newcommand{\PP}{\ensuremath{\mathbb{P}}}

\newcommand{\ZZ}{\ensuremath{\mathbb{Z}}}

\newcommand{\ra}{\ensuremath{\rightarrow}}

\def\hol{{\mathcal{O}}}

\newtheorem{teo}{Theorem}[section]
\newtheorem{lem}[teo]{Lemma}

\newtheorem{cor}[teo]{Corollary}
\newtheorem{prop}[teo]{Proposition}

\theoremstyle{definition}

\theoremstyle{remark}
\newtheorem{oss}[teo]{Remark}

\title[Rationality of moduli]{The rationality of certain  moduli spaces
  of curves of genus $3$} 

\author{Ingrid Bauer}
\address{Mathematisches Institut, Universit\"at Bayreuth, NW II, D-95440 Bayreuth; email:
Ingrid.Bauer@uni-bayreuth.de}
\author{Fabrizio Catanese}
\address{Mathematisches Institut, Universit\"at Bayreuth, NW II, D-95440 Bayreuth; email:
Fabrizio.Catanese@uni-bayreuth.de}
\date{\today}

\begin{document}

\thanks{The research of the authors was performed in the realm of the DFG Forschergruppe 790 
"Classification of algebraic surfaces and compact complex manifolds"}

\maketitle
\pagestyle{myheadings}

\section*{Introduction} The aim of this paper is to give an explicit geometric description of the
birational structure of the moduli space of pairs $(C,\eta)$, where $C$ is a general curve of
genus $3$ over an algebraically closed field $k$ of arbitrary characteristic and $\eta \in
Pic^0(C)_3$ is a non trivial divisor  class of  3-torsion on $C$.

As it was observed in \cite{torelli} lemma (2.18), if $C$ is a general curve of genus $3$ and
$\eta \in Pic^0(C)_3$ is a non trivial $3$ - torsion divisor class, then we have a morphism
$\varphi_{\eta} := \varphi_{|K_C +
\eta|} \times \varphi_{|K_C - \eta|} : C \ra \PP^1 \times \PP^1$, corresponding to
 the sum of the linear systems $|K_C +
\eta|$ and $|K_C - \eta|$, which is  birational onto a curve $\Gamma \subset \PP^1 \times \PP^1$
of bidegree
$(4,4)$. Moreover, $\Gamma$  has exactly  six ordinary double points as singularities,
 located in the six points of the set $\mathcal{S}:=\{(x,y)|x \neq y, ~ x,y \in \{0, 1, \infty \}
\}$.

In \cite{torelli} we only gave an outline of the proof (and there is also a minor inaccuracy).
Therefore we dedicate the first section of this article to a detailed geometrical description of
such pairs $(C, \eta)$, where
$C$ is a general curve of genus $3$ and $\eta \in Pic^0(C)_3 \setminus \{0\}$. 

The main result of the first section is the following: 

\begin{teo}\label{thm1} Let $C$ be a general (in particular, non hyperelliptic) curve of genus $3$
over an algebraically closed field $k$ (of arbitrary characteristic) and $\eta \in Pic^0(C)_3
\setminus \{0\}$. 

Then the rational map $\varphi_{\eta} : C \rightarrow \PP^1 \times \PP^1$ defined by 
$$\varphi_{\eta} := \varphi_{|K_C +
\eta|} \times \varphi_{|K_C - \eta|} : C \ra \PP^1 \times \PP^1$$ is a morphism, birational onto
its image 
$\Gamma$, which is a curve of bidegree $(4,4)$ having exactly six ordinary double points as
singularities. We can assume, up to composing  $\varphi_{\eta}$ with a transformation of $\PP^1
\times \PP^1$ in $ \PP GL(2, k)^2$ ,  that the singular set of $\Gamma$ is the set 
$$\mathcal{S}:=\{(x,y) \in \PP^1 \times
\PP^1| x \neq y \ ; x,y \in \{0,1, \infty \}\}.$$

Conversely, if $\Gamma$ is a curve of bidegree $(4,4)$ in $\PP^1 \times \PP^1$, whose
singularities consist of exactly six ordinary double points at the points of $\mathcal{S}$, its
normalization $C$ is a curve of genus $3$, s.t. $\hol_C(H_2 - H_1) =: \hol_C(\eta)$ (where $H_1$,
$H_2$ are the respective pull backs of the rulings of
$\PP^1 \times \PP^1$) yields a non trivial $3$ - torsion divisor class, and
$\hol_C(H_1) \cong \hol_C(K_C + \eta)$, $\hol_C(H_2) \cong \hol_C(K_C - \eta)$.
\end{teo}

From theorem (\ref{thm1}) it follows that 

$\mathcal{M}_{3,\eta}:= \{(C, \eta): C$ is a general curve of genus $3, ~ \eta \in Pic^0(C)_3
\setminus \{0\}\}$ is birational to $\PP(V(4,4, - \mathcal{S})) /\mathfrak{S}_3$, where 
$$ V(4,4, - \mathcal{S}):= H^0(\hol_{\PP^1 \times \PP^1}(4,4)(- 2\sum_{a\neq b, a,b \in \{\infty,
0,1\}} (a,b)).
$$   In fact, the permutation action of the symmetric group 

$\mathfrak{S}_3 := \mathfrak{S}(\{\infty, 0,1\})$ extends to an action on $\PP^1$, so
$\mathfrak{S}_3$ is naturally a subgroup of $\PP GL(2,k)$.  We consider then the diagonal action
of $\mathfrak{S}_3$ on $\PP^1 \times \PP^1$, and observe that $\mathfrak{S}_3$ is exactly the
subgroup of $\PP GL(2,k)^2$ leaving the set $\mathcal S$ invariant.  The action of
$\mathfrak{S}_3$ on $V(4,4, - \mathcal{S})$ is naturally induced by the diagonal inclusion  
$\mathfrak{S}_3 \subset \PP GL(2,k)^2$ .

On the other hand, if we consider only the subgroup of order three of $Pic^0(C)$  generated by a
non trivial $3$ - torsion element $\eta$, we see from theorem (\ref{thm1}) that we have to allow
the exchange of $\eta$  with $- \eta$, which  corresponds to exchanging the two factors of $\PP^1
\times \PP^1$. 

Therefore 
$\mathcal{M}_{3,\langle \eta \rangle}:= \{(C, \langle \eta \rangle): C$ general curve of genus
$3, ~ \langle \eta \rangle \cong \mathbb{Z}/3 \mathbb{Z} \subset Pic^0(C)\}$ is birational to
$\PP(V(4,4, - \mathcal{S})/(\mathfrak{S}_3 \times\mathbb{Z}/2)$, where the action of the generator
$\sigma$ (of $\mathbb{Z}/2\mathbb{Z}$) on
$V(4,4, - \mathcal{S})$ is induced by the action on $\PP^1 \times \PP^1$ obtained by exchanging
the two coordinates.

Our main result is the following:

\begin{teo} Let $k$ be an algebraically closed field of arbitrary characteristic. We have:

1) the moduli space $\mathcal{M}_{3,\eta}$ is rational;

2) the moduli space $\mathcal{M}_{3,\langle \eta \rangle}$ is rational.
\end{teo}

One could obtain the above result  abstractly from the method of Bogomolov and Katsylo
 (cf. \cite{bogkat}), but we prefer to prove the theorem while   explicitly calculating the field
of invariant functions. It mainly suffices to
 decompose the vector representation of $\mathfrak{S}_3$ on
$V(4,4, - \mathcal{S})$ into irreducible factors. Of course, if the characteristic of $k$ equals
two or three, it is no longer possible to decompose the $\mathfrak{S}_3$ - module $V(4,4, -
\mathcal{S})$ as a direct sum of irreducible submodules. Nevertheless, we can write down the field
of invariants and see that it is rational.


\section{The geometric description of pairs $(C, \eta)$.} In this section we give a geometric
description of pairs $(C,
\eta)$, where $C$ is a general curve of genus $3$ and $\eta$ is a non trivial element of
$Pic^0(C)_3$, and we prove theorem (\ref{thm1}).

Let $k$ be an  algebraically closed field of arbitrary characteristic.   We recall the following
observation from \cite{torelli}, p.374.

\begin{lem} Let $C$ be a general  curve of genus $3$ and $\eta \in Pic^0(C)_3$ a non trivial
divisor class (i.e., $\eta$ is not linearly equivalent to $0$). Then the linear system $|K_C +
\eta|$ is base point free. This holds more precisely under the assumption that the canonical
system $|K_C|$ does not contain two divisors of the form $ Q + 3 P$, $Q + 3 P'$,  and where the
3-torsion divisor class $P-P' $  is the class of $\eta$. This condition for all such $\eta$ is
 in turn equivalent to the fact that  $C$ is either hyperelliptic  or it is non hyperelliptic but
the canonical image
$\Sigma$ of $C$ does not admit two inflexional tangents meeting in a point $Q$ of $\Sigma$. 
\end{lem}

\begin{proof} Note that
$P$ is a base point of the linear system $|K_C + \eta|$ if and only if
$$ H^0(C, \mathcal{O}_C(K_C + \eta)) = H^0(C, \mathcal{O}_C(K_C + \eta - P)).
$$ Since $dim H^0(C, \mathcal{O}_C(K_C + \eta)) = 2$ this is equivalent to
$$dim H^1(C, \mathcal{O}_C(K_C + \eta - P)) = 1.$$

Since $H^1(C, \mathcal{O}_C(K_C + \eta - P)) \cong H^0(C, \mathcal{O}_C(P -
\eta))^*$, this is equivalent to the existence of a point $P'$ such that $P - \eta
\equiv P'$ (note that we denote linear equivalence by the classical notation ``$\equiv$''.)
Therefore $3P \equiv 3P'$ and $P \neq P'$, whence in particular $H^0(C, \mathcal{O}_C(3P)) \geq
2$. By Riemann - Roch we have
$$ dim H^0(C, \mathcal{O}_C(K_C - 3P)) =
$$
$$ deg(K_C - 3P) + 1 - g(C) + dim H^0(C, \mathcal{O}_C(3P) \geq 1.
$$ In particular, there is a point $Q$ such that $Q \equiv K_C - 3P \equiv K_C - 3P'$. \\

Going backwards, we see that this condition is not only necessary, but sufficient. If $C$ is
hyperelliptic, then $  Q + 3 P, Q + 3 P' \in |K_C| $ hence $P, P'$  are Weierstrass points, whence
$ 2 P \equiv 2 P'$, hence $P- P'$ yields a divisor class  $\eta$ of 2-torsion,  contradicting the
nontriviality of $\eta$.

Consider now the canonical embedding of  $C$  as a plane quartic $\Sigma$.
 Our condition means, geometrically, that
$C$ has two inflection points $P$, $P'$, such that the tangent lines to these points intersect in
$Q
\in C$. 

We shall show now that the (non hyperelliptic) curves of genus three  whose canonical image is a
quartic $\Sigma$ with the above properties are contained in a five dimesnional family, whence are
special in the moduli space
$\mathcal{M}_3$ of curves of genus three.

Let now $p$, $q$, $p'$ be three non collinear points in
$\mathbb{P}^2$. The quartics in $\mathbb{P}^2$ form a linear system of dimension $14$. Imposing
that a plane quartic contains the point $q$ is one linear condition. Moreover, the condition that
the line containing $p$ and
$q$ has intersection multiplicity equal to 3   with the quartic in the point $p$ gives three
further  linear conditions.  Similarly for the point $p'$,   and it is easy to see that the above
seven linear conditions are independent. Therefore the linear subsystem of quartics $\Sigma$
having two inflection points 
$p$, $p'$, such that the tangent lines to these points intersect in $q \in \Sigma$ has dimension
$14 - 3 - 3 - 1 = 7$. The group of automorphisms of $\mathbb{P}^2$ leaving the three points $p$,
$q$, $p'$ fixed has
  dimension $2$ and therefore the above quartics give rise to a five dimensional algebraic subset
of
$\mathcal{M}_3$.

Finally, if the points $P, P', Q$ are not distinct, we have (w.l.o.g.) $ P = Q$ and a similar
calculation shows that we have a family of dimension $ 7 -  3 = 4$.
\end{proof}

Consider now the morphism 
$$
\varphi_{\eta} (:= \varphi_{|K_C + \eta|} \times \varphi_{|K_C - \eta|})  : C \rightarrow \PP^1
\times \PP^1,
$$ and denote by $\Gamma \subset \PP^1 \times \PP^1$ the image of $C$ under $\varphi_{\eta}$.

\begin{oss} 1) Since $\eta$ is non trivial,  either $\Gamma$ is of bidegree $(4,4)$, or
deg$\varphi_{\eta} = 2$ and $\Gamma$ is of bidegree $(2,2)$.
 In fact deg $\varphi_{\eta} = 4$ implies $\eta \equiv - \eta$.

2) We shall  assume in the following that $\varphi_{\eta}$ is birational, since otherwise $C$ is
either hyperelliptic (if $\Gamma$ is singular) or $C$ is a double cover of an elliptic curve
$\Gamma$
 (branched in $4$ points). 

In both cases $C$ lies in a $5$ - dimensional subfamily of the moduli space $\mathcal{M}_3$ of
curves of genus $3$.
\end{oss}

Let $P_1, \ldots, P_m$ be the (possibly infinitely near) singular points of $\Gamma$, and let
$r_i$ be the multiplicity in $P_i$ of the proper transform of $\Gamma$.  Then, denoting by  $H_1$,
respectively
$H_2$, the  divisors of a vertical, respectively of a horizontal line  in $\PP^1 \times \PP^1$, 
we have that $\Gamma
\in |4H_1 + 4H_2 - \sum_{i=1}^m r_i P_i|$. By adjunction, the canonical system of $\Gamma$ is cut
out by $|2H_1 + 2H_2 - \sum_{i=1}^m (r_i-1) P_i|$, and therefore 
$$ 4 = degK_C = \Gamma \cdot (2H_1 + 2H_2 - \sum_{i=1}^m (r_i-1) P_i) = 16 -
\sum_{i=1}^mr_i(r_i-1).
$$ Hence $\sum_{i=1}^mr_i(r_i-1) = 12$, and we have the following possibilities

\bigskip

\begin{tabular}{|c|c|c|}
\hline & $m$& $(r_1,\ldots, r_m)$\\
\hline\hline i) &1& (4) \\
\hline ii)&2 & (3,3) \\
\hline iii) & 4 & (3,2,2,2) \\
\hline iv) &6 & (2,2,2,2,2,2)\\
\hline
\end{tabular}

\bigskip

We will show now that for a general curve only the last case occurs, i.e., $\Gamma$ has exactly
$6$ singular points of multiplicity $2$.

We denote by $S$ the blow up of $\PP^1 \times \PP^1$ in $P_1, \ldots , P_m$, and let $E_i$ be the
exceptional divisor of the first kind, total transform of the point $P_i$.

We shall first show that the first case (i.e., $m = 1$) corresponds to the case $\eta \equiv 0$.

\begin{prop}\label{4uple} Let $\Gamma \subset \PP^1 \times \PP^1$ a curve  of bidegree $(4,4)$
having a point $P$ of multiplicity $4$, such that its normalization $C \in |4H_1 + 4H_2 - 4E|$ has
genus $3$ (here, $E$ is the exceptional divisor of the blow up of $\PP^1 \times \PP^1$ in $P$.)
Then 
$$
\mathcal{O}_C(H_1) \cong \mathcal{O}_C(H_2) \cong \mathcal{O}_C(K_C).
$$
\end{prop}

In particular, if $\Gamma = \varphi_{\eta}(C)$, (i.e., we are in the case $m=1$) then $\eta \equiv
0$.

\begin{oss} Let $\Gamma$ be as in the proposition. Then the rational map $\PP^1 \times \PP^1
\dashrightarrow
\PP^2$ given by $|H_1 + H_2 -E|$ maps $\Gamma$ to a plane quartic.  Viceversa, given a plane
quartic $C'$, blowing up two points $p_1, p_2 \in (\PP^1 \times \PP^1) \setminus C'$,  and then
contracting the strict transform of the line through $p_1, p_2$, yields a curve $\Gamma$ of
bidegree $(4,4)$  having a singular point of multiplicity $4$.
\end{oss}

\begin{proof}[Proof (of the proposition)]  Let $H_1$ be the full transform of a vertical line
through $P$. Then there is an effective divisor $H_1'$ on the blow up $S$ of $\PP^1 \times \PP^1$
in $P$
 such that $H_1 \equiv H_1' + E$. Since $H_1 \cdot C = E \cdot C = 4$, $H_1'$ is disjoint from
$C$,  whence $\mathcal{O}_C(H_1) \cong
\mathcal{O}_C(E)$. The same argument for a horizontal line through $P$  obviously shows that
$\mathcal{O}_C(H_2)
\cong \mathcal{O}_C(E)$. If $h^0(C, \mathcal{O}_C(H_1)) = 2$,  then the two projections $p_1, p_2
: \Gamma \ra \PP^1$ induce the same linear series on $C$, thus $\varphi_{|H_1|}$ and
$\varphi_{|H_2|}$ are related by a projectivity of $\PP^1$, hence $\Gamma$ is the graph of a
projectivity of
$\PP^1$, contradicting the fact that the bidegree of $\Gamma$ is $(4,4)$.
 
Therefore we have a smooth curve of genus three and a divisor of degree 4 such that $h^0(C,
\mathcal{O}_C(H_1)) \geq 3$. Hence $h^0(C, \mathcal{O}_C(K_C - H_1)) \geq 1$,
 which implies that
$K_C \equiv H_1$. Analogously, $K_C \equiv H_2$.
\end{proof}

The next step is to show that for a general curve $C$ of genus $3$,  cases $ii)$ and $iii)$ do not
occur. In fact, we show:

\begin{lem} Let $C$ be a curve of genus $3$  and $\eta \in Pic^0(C)_3 \setminus \{0\}$ such that
$\varphi_{\eta}$ is birational and the image $\varphi_{\eta}(C) = \Gamma$ has a  singular point
$P$ of multiplicity $3$. Then $C$ belongs to an algebraic subset of $\mathcal{M}_3$ of dimension
$\leq 5$.
\end{lem}

\begin{proof} Let $S$ again be the blow up of $\PP^1 \times \PP^1$ in $P$,  and denote by $E$ the
exceptional divisor. Then $\mathcal{O}_C(E)$ has degree $3$ and arguing
 as in prop. (\ref{4uple}), we see that there are points $Q_1, Q_2$ on $C$ such that
$\mathcal{O}_C(H_i) \cong \mathcal{O}_C(Q_i+E)$.  Therefore
$\mathcal{O}_C(Q_2-Q_1) \cong \mathcal{O}_C(H_2-H_1) \cong 
\mathcal{O}_C(K_C - \eta -(K_C + \eta)) \cong
\mathcal{O}_C(\eta)$, whence $3Q_1 \equiv 3Q_2$, $Q_1 \neq Q_2$.
 This implies that there is a morphism $f : C
\rightarrow \PP^1$ of degree $3$, having double ramification in $Q_1$ and $Q_2$.
 By Hurwitz' formula the degree of the ramification divisor $R$  is $10$ and since $ R \geq Q_1 +
Q_2$ $f$ has at most 
$8$ branch points in $\PP^1$. Fixing three of these points to be $\infty, 0, 1$, we obatain (by
Riemann's existence theorem) a finite number of 
 families of dimension at most $5$.
\end{proof}

From now on, we shall make the following 

\noindent {\bf Assumptions.}

\noindent
$C$ is a curve of genus $3$, $\eta \in Pic^0(C)_3 \setminus \{0\}$, and 

\begin{itemize}
\item[1)] $|K_C + \eta|$ and $|K_C - \eta|$ are base point free;
\item[2)] $\varphi_{\eta} : C \rightarrow \Gamma \subset \PP^1 \times \PP^1$ is birational;
\item[3)] $\Gamma \in |4H_1 + 4H_2|$ has only  double points as singularities (possibly infinitely
near).
\end{itemize}

\begin{oss} By the considerations so far, we know that a general curve of genus $3$ fulfills the
assumptions for any $\eta \in Pic^0(C)_3 \setminus \{0\}$.
\end{oss}

We use the notation introduced above: we have $\pi : S \rightarrow \PP^1 \times \PP^1$ and $ C
\subset S$, $ C \in | 4 H_1 + 4 H_2 - 2 \sum_{i=1}^6 E_i |$.

\begin{oss}\label{Ramanujam} Since $S$ is a regular surface, we have an easy case of Ramanujam's
vanishing theorem: if $D$ is an effective divisor which is 1-connected (i.e., for every
decomposition
$ D = A + B$ with $  A, B > 0$, we have $ A \cdot B \geq 1$), then
$ H^1 ( S, \hol_S ( - D)) = 0$.

This follows immediately from Ramanujam's lemma ensuring that 
$H^0 ( D, \hol_D ) = k$, and the long exact cohomology sequence associated to
$$ 0 \ra  \hol_S ( - D) \ra \hol_S \ra \hol_D  \ra 0.$$ In most of our applications we shall show
that $D$ is linearly equivalent to a reduced and  connected divisor (this is a stronger property
than 1-connectedness).

\end{oss}
\qed

We know now that $\mathcal{O}_C(H_1 + H_2) \cong \mathcal{O}_C(2K_C)$, i.e.,
$$
\mathcal{O}_C \cong \mathcal{O}_C(3H_1 + 3H_2 - \sum _{i=1}^6 2E_i).
$$

Since $h^1(S, \mathcal{O}_S(-H_1-H_2)) = 0$, the exact sequence 
\begin{multline} 0 \rightarrow \mathcal{O}_S(-H_1 - H_2) \rightarrow \mathcal{O}_S(3H_1 + 3H_2 -
\sum _{i=1}^6 2E_i) \rightarrow \\
\rightarrow \mathcal{O}_C(3H_1 + 3H_2 - \sum _{i=1}^6 2E_i) \cong \mathcal{O}_C \rightarrow 0,
\end{multline}

is exact on global sections. 

In particular, $h^0(S, \mathcal{O}_S(3H_1 + 3H_2 - \sum _{i=1}^6 2E_i)) = 1$.  We denote by $G$ 
the unique divisor in the linear system $|3H_1 + 3H_2 - \sum _{i=1}^6 2E_i|$.  Note that $C \cap G
= \emptyset$ (since $\mathcal{O}_C
\cong \mathcal{O}_C(G)$).

\begin{oss}\label{E} There is no effective divisor $\tilde{G}$ on $S$ such that $G = \tilde{G} +
E_i$, since otherwise $\tilde{G} \cdot C = -2$,  contradicting that $\tilde{G}$ and $C$ have no
common component.

This means that $G + 2 \sum _{i=1}^6 E_i$ is the total transform of a curve $ G' \subset \PP^1
\times \PP^1$ of bidegree (3,3).
\end{oss}

\begin{lem}
$h^0(G,\hol_G) = 3$, $h^1(G, \hol_G) = 0$.
\end{lem}

\begin{proof} Consider the exact sequence 
$$ 0 \ra \hol_S(K_S) \ra \hol_S(K_S + G) \ra \hol_G(K_G) \ra 0.
$$ Since $h^0(S, \hol_S(K_S)) = h^1(S, \hol_S(K_S)) = 0$, we get 
$$ h^0(S, \hol_S(K_S + G)) = h^0(G, \hol_G(K_G)).
$$ Now, $K_S + G \equiv H_1 + H_2 - \sum_{i=1}^6 E_i$, therefore $(K_S + G)\cdot C = -4$, whence
$h^0(G,
\hol_G(K_G)) \cong h^0(S, \hol_S(K_S + G)) = 0$. 

Moreover, $h^1(G, \hol_G(K_G)) = h^1(S, \hol_S(K_S + G)) + 1$, and by Riemann - Roch  we infer
that, since  $h^1(S,
\hol_S(K_S + G)) = h^0(S,
\hol_S(- G)) = 0$, that $h^1(S,
\hol_S(K_S + G)) = 2$.
\end{proof}

We will show now that $G$ is reduced, hence, by the above lemma,  we shall obtain that $G$ has
exactly $3$ connected components.

\begin{prop}
$G$ is reduced.
\end{prop}

\begin{proof} By remark (\ref{E}) it is sufficient to show that the image of $G$ in $\PP^1 \times
\PP^1$,
 which we denoted by $G'$, is reduced.

Assume that there is an effective divisor $A'$ on $\PP^1 \times \PP^1$ such that $3A' \leq G'$. 
We clearly have $A' \cap
\Gamma \neq \emptyset$ but, after blowing up the six points $P_1, \ldots , P_6$,  the strict
transforms of $A'$ and of $\Gamma$ are disjoint, whence $A'$ and $G'$ must intersect in one of the
$P_i$'s, contradicting remark (\ref{E}).

If $G'$ is not reduced, we may uniquely write
 $G' = 2D_1 + D_2$ with $D_1, D_2$ reduced and having no common component. Up to exchanging the
factors of $\PP^1 \times \PP^1$, we have the following two possibilities:
\begin{itemize}
 \item[i)] $D_1 \in |H_1 + H_2|$;
 \item[ii)] $D_1 \in |H_1|$.
\end{itemize} In the first case also $D_2 \in |H_1 + H_2|$ and its strict transform is disjoint
from $C$. Remark (\ref{E}) implies that $D_2$ meets $\Gamma$ in points which do not belong to
$D_1$, whence $D_2$ has double points where it intersects $\Gamma$. Since $D_2 \cdot \Gamma = 8$
we see that $D_2$ has two points of multiplicity $2$, a contradiction ($D_2$ has bidegree $(1,1)$).

Assume now that $D_1 \in |H_1|$. Then, since $2D_1 \cdot \Gamma = 8$, $D_1$ contains $4$ of the
$P_i$'s and $D_2$ passes through the other two, say $P_1, P_2$. This implies that for the strict
transform of $D_2$ we have:
$\hat{D}_2 \equiv H_1 + 3H_2 - 2E_1 - 2E_2$, whence $\hat{D}_2 \cdot C = 8$, a contradiction.
\end{proof}

We write now $G= G_1 + G_2 + G_3$ as a sum of its connected components,  and accordingly $G' =
G'_1 + G'_2 + G'_3$. 

\begin{lem} The bidegree of $G'_j$, ($j \in \{1,2,3\}$) is $(1,1)$. 

Up to renumbering 
$P_1, \ldots, P_6$ we have $G'_1 \cap G'_2 = \{P_1, P_2\}$, $G'_1 \cap G'_3 = \{P_3, P_4\}$ and
$G'_2 \cap G'_3 = \{P_5, P_6\}$.

More precisely, $ G_1 \in  |H_1 + H_2 - E_1 - E_2 - E_3 - E_4|$,
$ G_2 \in  |H_1 + H_2 - E_1 - E_2 - E_5 - E_6|$, $ G_3 \in  |H_1 + H_2 - E_3 - E_4 - E_5 - E_6|$.
\end{lem}

\begin{proof}
 Assume for instance that  $G'_1$ has bidegree $(1,0)$.  Then there is a subset $I \subset \{1,
\ldots, 6\}$ such that $G_1 = H_1 - \sum_{i \in I} E_i$. Since $G_1 \cdot C = 0$, it follows that
$|I| = 2$. But then $G_1 \cdot (G - G_1) = 1$,  contradicting the fact that  $G_1$ is a connected
component of $G$. 

Let $(a_j, b_j)$ be the bidegree of $G_j$: then
$a_j, b_j \geq 1$ since a reduced divisor of bidegree $(m,0)$ is not connected for
$ m \geq 2$. Since $\sum a_j = \sum b_j = 3$, it follows that $a_j = b_j = 1$. 

Writing now $G_j \equiv H_1 + H_2 -
\sum_{i=1}^6 \mu(j,i) E_i$ we obtain 
$$\sum_{j=1}^3 \mu(j,i)= 2, \ \sum_{i=1}^6 \mu(j,i)= 4, \ \sum_{i=1}^6 \mu(k,i)\mu(j,i)= 2$$ 
since $G_j \cdot C = 0$) and  $G_k \cdot G_j = 0$).
 We get the second claim of the lemma provided that we show: $\mu(j,i) = 1, \forall i,j$.

The first formula shows that if  $\mu(j,i) \geq 2$, then $\mu(j,i) = 2$ and $\mu(h,i) = 0$ for $ h
\neq j$. Hence  the second formula shows that 
$$ \sum_{h,k \neq j}\sum_{i=1}^6 \mu(j,i) (\mu(h,i)+ \mu(k,i)) \leq 2,$$ contradicting the third
formulae. 
\end{proof}

In the remaining part of the section we will show that each $G'_i$ consists of  the union of a
vertical and a horizontal line in $\PP^1 \times \PP^1$. 

Since $\hol_C(K_C + \eta) \cong \hol_C(H_1)$ and $\hol_C(K_C - \eta) \cong \hol_C(H_2)$ we get: 
$$
\hol_C(2H_2 - H_1) \cong \hol_C(K_C) \cong \hol_C(2H_1 + 2H_2 - \sum_{i=1}^6 E_i),
$$ whence the exact sequence
\begin{multline} 0 \rightarrow \mathcal{O}_S(-H_1 - 4H_2 + \sum_{i=1}^6 E_i) 
\rightarrow \mathcal{O}_S(3H_1 - \sum _{i=1}^6 E_i) \rightarrow \\
\rightarrow \mathcal{O}_C(3H_1 - \sum _{i=1}^6 E_i) \cong \mathcal{O}_C \rightarrow 0,
\end{multline}

\begin{prop}\label{6lines}
$H^1(S,\hol_S(-(H_1 + 4H_2 - \sum_{i=1}^6 E_i))) = 0$.
\end{prop}

\begin{proof}
 The result follows immediately by  Ramanujam's vanishing theorem, but we can also give an
elementary proof using   remark \ref{Ramanujam}. 

It suffices to show that the linear system $|H_1 + 4H_2 - \sum_{i=1}^6 E_i|$ contains a reduced
and connected divisor.

Note that 
$G_1 + |3H_2 - E_5 - E_6| \subset |H_1 + 4H_2 - \sum_{i=1}^6 E_i|$,  and that
$|3H_2 - E_5 - E_6|$ contains $|H_2 - E_5 - E_6| + |2H_2|$, if there is a line $H_2$ containing
$P_1, P_2$, else it contains $|H_2 - E_5| + |H_2 - E_6| + |H_2|$. Since $ G_1 \cdot H_2 =  G_1
\cdot (H_2 - E_5)= G_1 \cdot  (H_2  - E_6 )= G_1 \cdot  (H_2  -  E_5 - E_6 )= 1$, we have obtained
in both cases a reduced and connected divisor.

\end{proof}

\begin{oss} One can indeed show, using
$G_2 + |3H_2 - E_3 - E_4| \subset |H_1 + 4H_2 - \sum_{i=1}^6 E_i|$ and $G_3 + |3H_2 - E_1 - E_2|
\subset |H_1 + 4H_2 - \sum_{i=1}^6 E_i|$ 
 that  $|H_1 + 4H_2 -
\sum_{i=1}^6 E_i|$ has no fixed part, and then by Bertini's theorem, since 
$(H_1 + 4H_2 - \sum_{i=1}^6 E_i)^2 = 8 - 6 = 2 > 0$,  a general curve in $|H_1 + 4H_2 -
\sum_{i=1}^6 E_i|$ is irreducible. 
\end{oss}

In view of proposition \ref{6lines}  
 the above exact sequence (and the one where the roles of $H_1, H_2$ are exchanged) yields the
following:

\begin{cor} For $j \in \{1,2\}$ there is exactly one divisor $N_j \in |3H_j - \sum_{i=1}^6 E_i|$.
\end{cor}

By the uniqueness of $G$, we see that $G = N_1 + N_2$. Denote by $N'_j$  the curve in 
$\PP^1 \times \PP^1$ whose total transform is $N_j +\sum_{i=1}^6 E_i$. 

We have just seen that $G$ is the strict transform of three vertical and three horizontal lines in
$\PP^1
\times \PP^1$. Hence each connected component $G_j$ splits into the strict transform of a vertical
and a horizontal line. Since $G$ is reduced, the lines are distinct (and there are no infinitely
near points).

We can choose coordinates in $\PP^1 \times\PP^1$ such that $G'_1 = (\{\infty\} \times \PP^1) \cup
(\PP^1 \times \{\infty\})$, 
$G'_2 = (\{0\} \times \PP^1) \cup (\PP^1 \times \{0\})$ and $G'_3 = (\{1\} \times \PP^1) \cup
(\PP^1 \times \{1\})$.

\begin{oss}

The points $P_1, \dots , P_6$ are then the points of the set $\mathcal S$ previously defined.

 Conversely, consider in $\PP^1 \times \PP^1$ the set  $ \mathcal{S} : = \{P_1, \ldots,
 P_6 \} = (\{\infty, 0,1\} \times \{\infty, 0,1\}) \setminus \{(\infty,
 \infty), (0,0), (1,1) \} $. Let $\pi : S \ra \PP^1 \times \PP^1$
 be the blow up of the points $P_1, \ldots, P_6$ and suppose (denoting the
 exceptional divisor over $P_i$ by $E_i$) that $C \in |4H_1 + 4H_2 - \sum 2 E_i|$ is a smooth 
curve.
 Then $C$ has genus $3$, $\hol_C(3H_1) \cong \hol_C(\sum E_i) \cong \hol_C(3H_2))$.  Setting
$\hol_C(\eta):=\hol_C(H_2 - H_1)$, we obtain therefore $3 \eta \equiv 0$.
\end{oss}

It remains to show that $\hol_C(\eta)$ is not isomorphic to $\hol_C$.

\begin{lem}
 $\eta$ is not trivial.
\end{lem}

\begin{proof}
 Assume $\eta \equiv 0$. Then $\hol_C(H_1) \cong \hol_C(H_2)$ and, since 
$\Gamma$ has bidegree $(4,4)$, we argue as in the proof of proposition \ref{4uple}) that
$h^0(\hol_C(H_i)) \geq 3$, whence $\hol_C(H_i) \cong \hol_C(K_C)$. 

The same argument shows that the two peojections of $\Gamma$ to $\PP^1$ yield two different
pencils in the canonical system. It follows that the canonical map of $C$ factors as the
composition of 
$C \ra \Gamma \subset \PP^1 \times \PP^1 $ with the rational map $ \psi : \PP^1 \times \PP^1
\dashrightarrow \PP^2$ which blows up one point and  contracts the vertical and horizontal line
through it. Since $\Gamma$ has six singular points, the canonical map sends $C$ birationally onto
a singular quartic curve in $\PP^2$, absurd.
\end{proof}

\section{Rationality of the moduli spaces}

In this section we will use the geometric despcription of pairs $(C, \eta)$, where $C$ is a genus
$3$ curve and
$\eta$ a non trivial $3$ - torsion divisor class, and study the birational structure of their
moduli space.

More precisely, we shall prove the following

\begin{teo}\label{rat} 1) The moduli space $\mathcal{M}_{3,\eta} := \{(C, \eta): C$ a general
curve of genus $3, ~ \eta \in Pic^0(C)_3 \setminus \{0\}\}$ is rational.

2) The moduli space $\mathcal{M}_{3,\langle \eta \rangle } := \{(C, \langle \eta \rangle):
 C$ a general curve of genus $3, ~ \langle \eta \rangle \cong
\mathbb{Z}/3 \mathbb{Z} \subset Pic^0(C)\}$ is rational.

\end{teo}

\begin{oss} By the result of the previous section, and since any automorphism of $\PP^1 \times
\PP^1 $ which sends the set $\mathcal S$ to itself belongs to the group
$\mathfrak{S}_3 \times \ZZ / 2 \ZZ$, follows immediately   that, if we set  
$$ V(4,4, - \mathcal{S}):= H^0(\hol_{\PP^1 \times \PP^1}(4,4)(- 2\sum_{i\neq j, i,j \in \{\infty,
0,1\}} P_{ij})),
$$

$\mathcal{M}_{3,\eta}$ is birational to
$\PP(V(4,4, - \mathcal{S})) /\mathfrak{S}_3$, 
 while  $\mathcal{M}_{3,\langle \eta \rangle }$ is birational to 
$\PP(V(4,4, - \mathcal{S}))/(\mathfrak{S}_3 \times \mathbb{Z}/2 \ZZ)$, where the generator
$\sigma$ of $\mathbb{Z}/2 \mathbb{Z}$ acts by coordinate exchange  on  $\PP^1 \times \PP^1 $,
whence on
$V(4,4, - \mathcal{S})$.
\end{oss}

In order to prove the above theorem we will explicitly calculate the  respective subfields of 
invariants of the function field  of $\PP(V(4,4, - \mathcal{S}))$ and show that they are generated
by purely transcendental elements.

Consider the following polynomials of $\mathbb{V}:=V(4,4, - \mathcal{S})$,  which are invariant
under the action of
$\mathbb{Z}/2 \ZZ$:
$$ f_{11} (x,y):= x^2_0x^2_1y^2_0y^2_1,
$$ 
$$ f_{\infty \infty} (x,y):= x^2_1(x_1 - x_0)^2y^2_0(y_1 - y_0)^2,
$$ 
$$ f_{00} (x,y):= x^2_0(x_1-x_0)^2y^2_0(y_1-y_0)^2.
$$   Let $ev: \mathbb{V} \ra \bigoplus_{i=0,1, \infty} k_{(i, i)} =: \mathbb{W}$ be the evaluation
map at the three standard diagonal points, i.e.,
$ev(f):=(f(0,0), f(1,1), f(\infty, \infty))$.

Since $f_{ii} (j,j) = \delta_{i,j}$, 
 we can decompose $\mathbb{V} \cong \mathbb{U} \oplus
\mathbb{W}$, where
$\mathbb{U}:= ker (ev)$ and $\mathbb{W}$ is the  subspace generated by the three above
polynomials, which is easily shown to be an invariant subspace using the following formulae $(*)$:

\begin{itemize}
\item $(1,3)$ exchanges $x_0$ with $x_1$, multiplies $x_1 - x_0$ by $-1$, \\
\item $(1,2)$ exchanges $x_1 - x_0$ with $x_1$, multiplies $x_0$ by $-1$,\\
\item $(2,3)$ exchanges $x_0 - x_1$ with $x_0$, multiplies $x_1$ by $-1$.  
\end{itemize}

 In fact, `the permutation' representation $\mathbb{W}$ of the symmetric group  splits (in
characteristic $\neq 3$) as the direct sum of the trivial representation (generated by $e_1 + e_2
+ e_3$) and the standard representation, generated by $ x_0 : = e_1 - e_2, x_1 : = -e_2 + e_3$, 
which is isomorphic to the representation on $ V(1) : = H^0 (\hol_{\PP^1}(1))$.

Note that $\mathbb{U} = x_0x_1(x_1-x_0)y_0y_1(y_0-y_1)H^0(\PP^1 \times
\PP^1,\hol_{\PP^1 \times \PP^1}(1,1))$. 

\noindent We write $V(1,1):=H^0(\PP^1 \times
\PP^1,\hol_{\PP^1 \times \PP^1}(1,1)) = V(1) \otimes V(1)$, where 
$V(1):= H^0(\PP^1,\hol_{\PP^1}(1))$,  is  as above the standard representation of
$\mathfrak{S}_3$. 

 Now $V(1) \otimes V(1)$ splits, in characteristic $ \neq 2,3$, as a sum  of irreducible
representations  $\mathbb{I} \oplus  \mathfrak{A} \oplus W$, where the three factors are the {\em
trivial}, the {\em alternating} and the {\em standard} representation of $\mathfrak{S}_3$.

Explicitly,    $V(1) \otimes V(1) \cong \wedge^2 (V(1)) \oplus Sym^2(V(1))$, and $Sym^2(V(1))$ is
isomorphic to $\mathbb{W}$, since it has the following basis
$ x_0 y_0, x_1 y_1, (x_1 - x_0) (y_1 - y_0)$. We observe for further use that $\ZZ / 2 \ZZ$ acts
as the identity on $Sym^2(V(1))$, while it acts on $\wedge^2 (V(1))$, spanned by $x_1 y_0 - x_0
y_1$ via multiplication by $-1$.

We have thus seen
\begin{lem} If $char(k) \neq 2,3$ then the $\mathfrak{S}_3$ -module $\mathbb{V}$ splits as a sum
of irreducible modules as follows:
$$ \mathbb{V} \cong  2(\mathbb{I} \oplus W) \oplus \mathfrak{A}.
$$
\end{lem}

Choose now a basis $(z_1,z_2,z_3,w_1,w_2,w_3,u)$ of $\mathbb{V}$,  such that the $z_i$'s and the
$w_i$'s are respective
 bases of $\mathbb{I} \oplus W$ consisting of eigenvectors of $\sigma = (123)$, and $u$ is a basis
element of  $\mathfrak{A}$. The eigenvalue of $z_i, w_i$ with respect to $\sigma = (123)$ is 
$\epsilon^{i-1}$, $u$ is $\sigma$-invariant and $(12)(u) = -u$.

Note that if $(v_1,v_2,v_3)$ is a basis of $\mathbb{I} \oplus W$, such that
$\mathfrak{S}_3$ acts by permutation of the indices, then 
$z_1 = v_1 + v_2 + v_3$, $z_2 = v_1 + \epsilon v_2 +
\epsilon^2 v_3$, $z_3 = v_1 + \epsilon^2 v_2 +
\epsilon v_3$, where $\epsilon$ is a primitive third root of unity.

\begin{oss} Since $z_1,w_1$ are $\mathfrak{S}_3$ - invariant, $\PP(V(4,4, - \mathcal{S}))
/\mathfrak{S}_3$ is birational to a product of the affine line with
$Spec(k[z_2,z_3,w_2,w_3,u]^{\mathfrak{S}_3})$, and therefore it suffices to compute
$k[z_2,z_3,w_2,w_3,u]^{\mathfrak{S}_3}$.
\end{oss}

Part 1 of the theorem follows now from the following

\begin{prop} Let $T:=z_2z_3$, $S:=z_2^3$, $A_1:=z_2w_3+z_3w_2$, $A_2:=z_2w_3-z_3w_2$. Then
\begin{multline*} k(z_2,z_3,w_2,w_3,u)^{\mathfrak{S}_3} 
\supset K:= \\ k(A_1,T,S+\frac{T^3}{S},u(S-\frac{T^3}{S}), A_2(S-\frac{T^3}{S})),
\end{multline*}

and $[k(z_2,z_3,w_2,w_3,u):K] = 6$, hence $ k(z_2,z_3,w_2,w_3,u)^{\mathfrak{S}_3} = K$.
\end{prop}

\begin{proof} We first calculate the invariants under the action of $\sigma = (123)$, i.e.,
$k(z_2,z_3,w_2,w_3,u)^{\sigma}$. Note that $u$, $z_2z_3$, $z_2w_3$,
$w_2w_3$,$z_2^3$ are $\sigma$-invariant, and
$[k(z_2,z_3,w_2,w_3,u):k(u, z_2z_3, z_2w_3,w_2w_3,z_2^3)] =3$.
 In particular, $k(z_2,z_3,w_2,w_3,u)^{\sigma} = k(u, z_2z_3, z_2w_3,w_2w_3,z_2^3) =: L$.

Now, we calculate $L^{\tau}$, with $\tau = (12)$. We first observe that 
$L= k(T,A_1,A_2,S,u)$. Since $\tau(z_2) =
\epsilon z_3$, $\tau(z_3) = \epsilon^2 z_2$ (and similarly for $w_2$, $w_3$), we see that
$\tau(A_1) = A_1$ and
$\tau(T) = T$. On the other hand, $\tau(u) = -u$, $\tau(A_2) = -A_2$, $\tau(S) = \frac{T^3}{S}$. 

\noindent {\bf Claim.}

\noindent
$L^{\tau} = k(A_1,T,S+\frac{T^3}{S},u(S-\frac{T^3}{S}), A_2(S-\frac{T^3}{S})) =: E$.

\medskip

\noindent {\it Proof of the Claim.}  Obviously
  $A_1$,$T$,$S+\frac{T^3}{S}$,$u(S-\frac{T^3}{S})$, $A_2(S-\frac{T^3}{S})$ are
  invariant under $\tau$, whence $E \subset L^{\tau}$. Since $L = E(S)$, using the equation $B
  \cdot S = S^2 + T^3$ for $B:= S+\frac{T^3}{S}$, we get that $[E(S):E] \leq 2$.

This proves the claim and the proposition.
\end{proof}

There remains to show the second part of the theorem. 

We denote by $\tau'$ the involution on
$k(z_1,z_2,z_3,w_1,w_2,w_3,u)$ induced by the involution $(x,y) \mapsto (y,x)$ on $\PP^1 \times
\PP^1$. It suffices to prove the following
\begin{prop}\label{sigma}
$E^{\tau'} = k(A_1,T,S+\frac{T^3}{S},(u(S-\frac{T^3}{S}))^2, A_2(S-\frac{T^3}{S}))$.
\end{prop}

\begin{proof} Since $[E:k(A_1,T,S+\frac{T^3}{S},(u(S-\frac{T^3}{S}))^2, A_2(S-\frac{T^3}{S}))]
\leq 2$, it suffices to show that the $5$ generators
$A_1$, $T$,$S+\frac{T^3}{S}$,$(u(S-\frac{T^3}{S}))^2$, $A_2(S-\frac{T^3}{S})$ are $\tau '$ -
invariant.  This will now  be proven in lemma (\ref{tau'}).
\end{proof}

\begin{lem}\label{tau'} $\tau'$ acts as the identity on
$(z_1,z_2,z_3,w_1,w_2,w_3)$ and sends $u \mapsto -u$.
\end{lem}

\begin{proof} We note first that $\tau '$ acts trivially on  the subspace $\mathbb{W}$ generated
by the polynomials $ f_{ii}$. 

Since $\mathbb{U} = x_0x_1(x_1-x_0)y_0y_1(y_1-y_0) V(1,1)$ and $x_0x_1(x_1-x_0)y_0y_1(y_1-y_0)$
 is invariant under exchanging $x$ and
$y$, it suffices to recall that  the action of
$\tau '$ on $V(1,1) = V(1) \otimes V(1)$ is the identity on the subspace $ Sym^2 (V(1))$, while
the action on the alternating
$\mathfrak{S}_3$ - submodule $\mathfrak{A}$ sends the generator $u$  to $-u$. 
\end{proof}

\subsection{$Char(k) = 3$}

In order to prove theorem (\ref{rat}) if the characteristic of $k$ is equal to $3$ we describe the
$\mathfrak{S}_3$ -module $\mathbb{V}$ as follows:
$$ \mathbb{V} \cong  2\mathbb{W} \oplus \mathfrak{A},
$$ where $\mathbb{W}$ is the ($3$-dimensional) permutation representation of $\mathfrak{S}_3$.

Let now $z_1, z_2,z_3,w_1,w_2,w_3,u$ be a basis of $\mathbb{V}$ such that the action of
$\mathfrak{S}_3$ permutes $z_1,z_2,z_3$ (resp. $w_1,w_2,w_3$), and $(123) : u \mapsto u$, $(12) Ö
u \mapsto -u$. Then we have:

\begin{prop}
 The $\mathfrak{S}_3$- invariant subfield $k(\mathbb{V})^{\mathfrak{S}_3}$ of $k(\mathbb{V})$ is
rational.

More precisely, the seven $\mathfrak{S}_3$ - invariant functions 
$$\sigma_1 = z_1 + z_2 + z_3,$$
$$\sigma_2 = z_1z_2 + z_1z_3+z_2z_3,$$
$$\sigma_3 = z_1 z_2 z_3,$$
$$\sigma_4 = z_1w_1 + z_2w_2+z_3w_3,$$
$$\sigma_5 = w_1z_2z_3 + w_2z_1z_3 + w_3z_1z_2,$$
$$\sigma_6 = w_1(z_2+z_3) + w_2(z_1+z_3) + w_3(z_1+z_2),$$
$$\sigma_7 = u(z_1(w_2-w_3) + z_2(w_3-w_1) + z_3(w_1-w_2))$$ form a basis of the purely
transcendental extension over $k$.
\end{prop}

\begin{proof}
 $\sigma_1, \ldots , \sigma_7$ determine a morphism $\psi : \mathbb{V} \rightarrow
\mathbb{A}^7_k$. We will show that $\psi$ induces a birational map $\bar{\psi} : \mathbb{V} /
\mathfrak{S}_3 \rightarrow \mathbb{A}^7_k$, i.e., for a Zariski open set of $\mathbb{V}$ we have:
$\psi(x) = \psi(x')$ if and only if there is a $\tau \in \mathfrak{S}_3$ such that $x=\tau(x')$.
By \cite{fabrat}, lemma (2.2), we can assume (after acting on $x$ with a suitable $\tau \in
\mathfrak{S}_3$), that $x_i = x_i'$ for $1 \leq i \leq 6$, and we know that (setting $u:=x_7$,
$u':= x_7'$) 
\begin{multline*} u(x_1(x_5-x_6) + x_2(x_6-x_4)+x_3(x_4-x_5)) = \\ u' (x_1(x_5-x_6) +
x_2(x_6-x_4)+x_3(x_4-x_5)).
\end{multline*} Therefore, if $B(x_1, \ldots , x_6):= x_1(x_5-x_6) + x_2(x_6-x_4)+x_3(x_4-x_5)
\neq 0$, this implies that $u = u'$.
\end{proof}

Therefore, we have shown part 1 of theorem (\ref{rat}). 

We denote again by $\tau'$ the involution on
$k(z_1,z_2,z_3,w_1,w_2,w_3,u)$ induced by the involution $(x,y) \mapsto (y,x)$ on $\PP^1 \times
\PP^1$. In order to prove part 2) of thm. (\ref{rat}), it sufficies to observe that $\sigma_1,
\ldots, \sigma_6, \sigma_7^2$ are invariant under $\tau'$ and $[k(\sigma_1, \ldots , \sigma_7) :
k(\sigma_1, \ldots , \sigma_7^2)] \leq 2$, whence
$(k(\mathbb{V})^{\mathfrak{S}_3})^{(\mathbb{Z}/2\mathbb{Z})} = k(\sigma_1, \ldots , \sigma_7^2)$.
This proves theorem (\ref{rat}).

\subsection{$Char(k) = 2$} Let $k$ be an algebraically closed field of characteristic $2$. Then we
can describe the $\mathfrak{S}_3$ - module $\mathbb{V}$ as follows:
$$
\mathbb{V} \cong  \mathbb{W} \oplus V(1,1),
$$ where $\mathbb{W}$ is the ($3$ - dimensional) permutation representation of $\mathfrak{S}_3$.
We denote a basis of $\mathbb{W}$ by $z_1,z_2,z_3$. As in the beginning of the chapter,  $V(1,1) =
H^0(\PP^1 \times
\PP^1,\hol_{\PP^1 \times \PP^1}(1,1))$. We choose the following basis of $V(1,1)$: $w_1:=x_1y_1$,
$w_2:= (x_0 + x_1)(y_0 + y_1)$, $w_3:=x_0y_0$, $w:=x_0y_1$. Then $\mathfrak{S}_3$ acts on $w_1,
w_2, w_3$ by permutation of the indices and 
$$ (1,2): w \mapsto w+w_3, 
$$
$$ (1,2,3): w \mapsto w+w_2+w_3.
$$ Let $\epsilon \in k$ be a nontrivial third root of unity. Then Theorem (\ref{rat}) (in
characteristic $2$) follows from the following result:

\begin{prop}\label{invchar2} Let $k$ be an algebraically closed field of characteristic $2$. Let
$\sigma_1, \ldots, \sigma_6$ be as defined in (\ref{sigma}) and set 
$$ v:= (w+w_2)(w_1+\epsilon w_2 + \epsilon ^2 w_3) + (w+w_1 + w_3)(w_1+\epsilon^2 w_2 + \epsilon
w_3),
$$
$$ t:=(w+w_2)(w+w_1 + w_3).
$$ Then
\begin{itemize}
\item[1)] $k(z_1,z_2,z_3,w_1,w_2,w_3,w)^{\mathfrak{S}_3} = k(\sigma_1, \ldots , \sigma_6,v)$; \\
\item[2)] $k(z_1,z_2,z_3,w_1,w_2,w_3,w)^{\mathfrak{S}_3 \times \mathbb{Z}/2 \mathbb{Z}} =
k(\sigma_1, \ldots , \sigma_6,t)$.
\end{itemize}
\end{prop}

In particular, the respective invariant subfields of $k(\mathbb{V})$ are generated by purely
transcendental elements, and this proves theorem (\ref{rat}).

\begin{proof}[Proof of (\ref{invchar2}).] 2) We observe that $\mathbb{Z}/2 \mathbb{Z}$ ($x_i
\mapsto y_i$) acts trivially on $z_1,z_2,z_3,w_1,w_2,w_3$ and maps $w$ to $w + w_1+w_2+w_3$. It is
now easy to see that $t$ is invariant under the action of $\mathfrak{S}_3 \times \mathbb{Z}/2
\mathbb{Z}$. Therefore $k(\sigma_1, \ldots , \sigma_6,t) \subset
K:=k(z_1,z_2,z_3,w_1,w_2,w_3,w)^{\mathfrak{S}_3 \times \mathbb{Z}/2 \mathbb{Z}}$. By
\cite{fabrat}, lemma (2.8), $[k(z_1,z_2,z_3,w_1,w_2,w_3,t):k(\sigma_1, \ldots , \sigma_6,t)] = 6$,
and obviously, $[k(z_1,z_2,z_3,w_1,w_2,w_3,w): k(z_1,z_2,z_3,w_1,w_2,w_3,t)] = 2$. Therefore
$[k(z_1,z_2,z_3,w_1,w_2,w_3,w):k(\sigma_1, \ldots , \sigma_6,t)] = 12$, whence $K = k(\sigma_1,
\ldots , \sigma_6,t)$.

1) Note that for $W_2:=w_1+\epsilon w_2 + \epsilon ^2 w_3$,  $W_3:=w_1+\epsilon^2 w_2 + \epsilon
^3 w_3$, we have: $W_2^3$ and $W_3^3$ are invariant under $(1,2,3)$ and are exchanged under
$(1,2)$. Therefore $v$ is invariant under the action of $\mathfrak{S}_3$ and we have seen that
$k(\sigma_1, \ldots , \sigma_6,v) \subset L:= k(z_1,z_2,z_3,w_1,w_2,w_3,w)^{\mathfrak{S}_3}$, in
particular $[k(z_1,z_2,z_3,w_1,w_2,w_3,w):k(\sigma_1, \ldots , \sigma_6,v)] \geq 6$. On the other
hand, note that $k(z_1,z_2,z_3,w_1,w_2,w_3,w) = k(z_1,z_2,z_3,w_1,w_2,w_3,v)]$ (since $v$ is
linear in $w$) and again, by \cite{fabrat}, lemma (2.8), $[k(z_i,w_i,v): k(\sigma_1, \ldots ,
\sigma_6,v)] = 6$. This implies that $L=k(\sigma_1, \ldots , \sigma_6,v)$. 
\end{proof}

\end{document}